\documentclass[11pt, a4paper, reqno]{amsart}

\usepackage{anysize}
\usepackage[utf8]{inputenc}
\usepackage{amsmath}
\usepackage{mathtools}
\usepackage{amssymb}
\usepackage{enumerate}
\usepackage{calc}
\usepackage{xcolor}
\usepackage{hyperref}
\usepackage{thmtools, thm-restate}

\marginsize{2.5cm}{2.5cm}{3cm}{3cm}

\newtheorem{corollary}{Corollary}
\newtheorem{lemma}{Lemma}

\newcommand{\C} {{\mathbb C}}      		    
\newcommand{\N} {{\mathbb N}}		        
\newcommand{\R} {{\mathbb R}}		        
\newcommand{\Z} {{\mathbb Z}}			    

\renewcommand{\O}{\mathcal{O}}			    

\renewcommand{\c}{\widetilde{c}} 		    

\newcommand{\defeq}{\coloneqq}              
\newcommand{\eref}[1]{(\ref{#1})}           
\newcommand{\sref}[1]{Section~\ref{#1}}     
\newcommand{\thmref}[1]{Theorem~\ref{#1}}   
\newcommand{\corref}[1]{Corollary~\ref{#1}} 
\newcommand{\lemref}[1]{Lemma~\ref{#1}}     

\begin{document}

\title[On Ramanujan Sums of a Real Variable and a New Ramanujan Expansion for the...]{On Ramanujan Sums of a Real Variable and a New Ramanujan Expansion for the Divisor Function}

\author{Matthew S. Fox}
\address{Department of Physics\\ Harvey Mudd College}
\email{msfox@hmc.edu}

\author{Chaitanya Karamchedu}
\address{Department of Mathematics\\ Harvey Mudd College}
\email{ckaramchedu@hmc.edu}

\maketitle

\begin{abstract}
We show that the absolute convergence of a Ramanujan expansion does not guarantee the convergence of its real variable generalization, which is obtained by replacing the integer argument in the Ramanujan sums with a real number. We also construct a new Ramanujan expansion for the divisor function. While our expansion is amenable to a continuous and absolutely convergent real variable generalization, it only interpolates the divisor function locally on $\R$.
\end{abstract}

\section{Introduction}
\label{sec:introduction}

In \cite{Ramanujan1}, Ramanujan introduced the exponential sums $c_q : \N \rightarrow \C$ defined by
\begin{equation}
c_{q}(n) \defeq \sum_{k \in (\Z/q\Z)^\times} e^{2\pi i k n/ q}
\label{eq:RamanujanSum}
\end{equation}
for any two natural numbers $q$ and $n$. These sums are known as \emph{Ramanujan sums} and have many remarkable properties. For example, $c_q(n)$ is integer-valued, $q$-even in the argument $n$, and multiplicative in the index $q$. See the surveys \cite{Lucht1, Murty1} and books \cite{Schwarz1, Sivaramakrishnan1} for details and more. Additionally, several important arithmetic functions can be expressed as a linear combination of Ramanujan sums of the form
\begin{equation}
f(n) = \sum_{q \geq 1} \widehat{f}(q) c_q(n)
\label{eq:RamanujanExpansion}
\end{equation}
for appropriate complex numbers $\widehat{f}(q)$. If this series converges, we say the arithmetic function $f(n)$ admits a \emph{Ramanujan expansion}. One famous example is the divisor function $\sigma_k(n) \defeq \sum_{d | n} d^k$, which admits the Ramanujan expansion
\begin{equation}
\sigma_k(n) = n^{k} \zeta(k + 1)\sum_{q \geq 1} \frac{c_{q}(n)}{q^{k + 1}}.
\label{eq:divisorFunction}
\end{equation}
This series converges absolutely when $k > 0$ since $|c_q(n)| \leq \sigma_1(n)$ for every $q$ \cite{Hardy1}. Note, however, that this and all other Ramanujan expansions are never unique since $\sum_{q \geq 1} c_q(n) / q = 0$ for all $n$ \cite{Ramanujan1}.

Several studies have generalized the notion of Ramanujan expansions to a wider class of functions. For example, Cohen has formulated a theory of Ramanujan expansions for unitary analogues of arithmetic functions \cite{Cohen2}, while T\'oth and Ushiroya have developed a generalization to multivariate arithmetic functions \cite{Toth1, Ushiroya1}. However, these generalizations all remain within the context of \emph{arithmetic} functions. It seems interesting, therefore, to formulate a theory of Ramanujan expansions for non-arithmetic functions, like functions of a real variable. In the closing to \cite{Murty1}, Murty suggests a simple way one might do this: ``The interesting thing about the right hand side of [\eref{eq:RamanujanExpansion}] is that if the series converges absolutely and we replace $n$ by a real number $x$, we obtain a continuous function which interpolates the given arithmetical function. In this way, we can view the Ramanujan expansion as a continuous analogue of the discretely defined arithmetical function.''

In this note, we briefly explore Murty's idea of generalizing \eref{eq:RamanujanSum} to the continuous sum $\c_q : \R \rightarrow \C$ defined by
\begin{equation}
\c_q(x) \defeq \sum_{k \in (\Z/q\Z)^{\times}} e^{2 \pi i k x / q}
\label{eq:GeneralizedRamanujanSumDefinition}
\end{equation}
for any real number $x$ and natural number $q$. Since $c_q(n) = \c_q(n)$ for all natural numbers $n$ and $q$, the function
\begin{equation}
\widetilde{f}(x) \defeq \sum_{q \geq 1} \widehat{f}(q) \c_q(x)
\label{eq:RamanujanExpansionCont}
\end{equation}
will indeed interpolate the arithmetic function $f(n)$ in \eref{eq:RamanujanExpansion}, provided \eref{eq:RamanujanExpansionCont} converges. The convergence properties of \eref{eq:RamanujanExpansionCont} are thus central to the viability of Murty's idea. In this paper we explore what, if anything, the absolute convergence of \eref{eq:RamanujanExpansion} implies about the limiting behavior of \eref{eq:RamanujanExpansionCont}.

Heuristically, one might expect $\c_q(x)$ to behave roughly like a sum of $|(\Z / q\Z)^{\times}| = \varphi(q)$ random complex numbers, where the real and imaginary parts are each uniformly distributed on the interval $[-1, 1] \subset \R$. Insofar as this approximation holds, one would expect $\c_q(x) = \O_x(1)$,\footnote{We use the asymptotic notation $f = \O_x(g)$ to denote the estimate $|f| \leq C_x g$ for some absolute constant $C_x > 0$ that may depend on the parameter $x$.} since the expectation value of the sum equals zero. Indeed, this behavior is also what one would naturally intuit from the asymptotic behavior of $c_q(n)$. In this context, therefore, it appears reasonable to adopt Murty's expectation that the absolute convergence of \eref{eq:RamanujanExpansion} implies the absolute convergence of \eref{eq:RamanujanExpansionCont}. 

Surprisingly, however, this intuition is flawed due to the following theorem (\sref{sec:sectionTwo}):

\begin{restatable}{theorem}{maintheorem}
\label{thm:maintheorem}
For fixed $x \in \R \backslash \Z$ and all $q > 2\pi|x|$,
$$
\c_q(x) = \frac{e^{2\pi i x} - 1}{2\pi i x} \varphi(q) + \O_x(\sigma_0(q)).
$$
\end{restatable}
In other words, $\c_q(x) = \O_x(\varphi(q))$ for all $q > 2\pi|x|$. In big-$\Omega$ notation, it follows that the real variable generalization $\widetilde{f}(x)$ of the arithmetic function $f(n)$ will diverge for all $x \in \R \backslash \Z$ whenever the Ramanujan coefficients $\widehat{f}(q)$ of $f(n)$ are $\Omega(q^{-2})$. As an example, we show that although the Ramanujan expansion \eref{eq:divisorFunction} for the sum-of-divisors function, $\sigma_1(n)$, converges absolutely, its real variable generalization does not:
\begin{corollary}
\label{cor:divisorCor}
The generalized sum-of-divisors function,
\begin{equation}
\widetilde{\sigma}_1(x) \defeq x\zeta(2)\sum_{q \geq 1} \frac{\c_q(x)}{q^{2}},
\label{eq:sumofdivisorcontinuous}
\end{equation}
diverges for all $x \in \R \backslash \Z$.
\end{corollary}
This raises the question, does there exist a continuous and absolutely convergent Ramanujan expansion that interpolates $\sigma_1(n)$? In \sref{sec:divisorExtension}, we establish that there does, but it is unnatural because it only interpolates $\sigma_1(n)$ in the open neighborhood $(n-1,n+1) \subset \R$. Consequently, our work leaves open the question of whether there exists a continuous and absolutely convergent Ramanujan expansion that interpolates $\sigma_1(n)$ \emph{globally} on $\R$, or at least on the positive reals.

\section{Proofs of Results}
\label{sec:sectionTwo}

The proof of \thmref{thm:maintheorem} relies on the following lemma:

\begin{lemma}
\label{lem:mainlemma}
If $x \in \R \backslash \Z$ and $d > 2 \pi |x|$, then 
$$
\frac{1}{e^{-2\pi i x / d} - 1} = -\frac{d}{2\pi i x} + \O_x(1).
$$
\end{lemma}
\begin{proof}
The result follows immediately from Taylor-expanding the exponential.
\end{proof}

\maintheorem*

\begin{proof}
The sum of the M\"obius function over the divisors of $n$, $\sum_{d | n}\mu(d)$, is one if $n = 1$ and zero otherwise. Hence, $\c_q(x)$ in \eref{eq:GeneralizedRamanujanSumDefinition} is the same as
$$
\c_q(x) = \sum_{k = 1}^{q} e^{2 \pi i k x / q} \sum_{d | (k,q)} \mu(d) = \sum_{d | q} \mu(d) \sum_{\substack{k = 1 \\ d | k}}^{q} e^{2\pi i k x / q} = \sum_{d | q} \mu(d) \sum_{l=1}^{q/d} e^{2\pi i l d x  / q}.
$$
Since $x \in \R\backslash\Z$, $\left|e^{2\pi i l d x / q}\right| < 1$, so we can apply the geometric series formula to the inner sum:
\begin{equation}
\c_q(x) = \left(1 - e^{2\pi i x}\right)\sum_{d | q} \frac{\mu(d)}{e^{-2\pi i d x / q} - 1} = \left(1 - e^{2\pi i x}\right)\sum_{d | q} \frac{\mu(q/d)}{e^{-2\pi i x / d} - 1}.
\label{eq:SumSeparated}
\end{equation}
Since $q > 2\pi|x|$, we can partition this sum into a sum over $d < 2\pi|x|$ and another over $d > 2\pi|x|$, 
$$
\c_q(x) =  \left(1 - e^{2\pi i x}\right) \sum_{\substack{d | q \\ d<2\pi|x|}} \frac{\mu(q/d)}{e^{-2\pi i x / d} - 1} +  \left(1 - e^{2\pi i x}\right) \sum_{\substack{d | q \\ d > 2\pi|x|}} \frac{\mu(q/d)}{e^{-2\pi i x / d} - 1} 
$$
\lemref{lem:mainlemma} allows us to write the $d > 2\pi|x|$ sum as
$$
\sum_{\substack{d | q \\ d>2\pi|x|}} \frac{\mu(q/d)}{e^{-2\pi i x / d} - 1} = -\frac{1}{2\pi i x}\sum_{\substack{d | q \\ d>2\pi|x|}} \mu\left(\frac{q}{d}\right) d + \sum_{\substack{d | q \\ d>2\pi|x|}} \mu\left(\frac{q}{d}\right) \O_x(1).
$$
Since
$$
\varphi(q) = \sum_{d | q} \mu\left(\frac{q}{d}\right)d = \sum_{\substack{d | q \\ d<2\pi|x|}} \mu\left(\frac{q}{d}\right)d + \sum_{\substack{d | q \\ d>2\pi|x|}} \mu\left(\frac{q}{d}\right)d
$$
and, similarly,
$$
\O_x(\sigma_0(q)) = \sum_{d | q} \mu\left(\frac{q}{d}\right)\O_x(1) = \sum_{\substack{d | q \\ d<2\pi|x|}} \mu\left(\frac{q}{d}\right)\O_x(1) + \sum_{\substack{d | q \\ d>2\pi|x|}} \mu\left(\frac{q}{d}\right) \O_x(1)
$$
it follows that
\begin{equation}
\sum_{\substack{d | q \\ d>2\pi|x|}} \frac{\mu(q/d)}{e^{-2\pi i x / d} - 1} = -\frac{1}{2\pi i x}\varphi(q) + \O_x(\sigma_0(q)) - \frac{1}{2\pi i x}\sum_{\substack{d|q\\ d < 2\pi|x|}} \mu\left(\frac{q}{d}\right) d - \sum_{\substack{d|q\\ d < 2\pi|x|}} \mu\left(\frac{q}{d}\right) \O_x(1).
\label{eq:d>rSum}
\end{equation}
Now consider the $d < 2\pi|x|$ sum in \eref{eq:SumSeparated}. Since $q$ has at most $\lfloor 2\pi|x| \rfloor$ divisors less than $2\pi|x|$,
$$
\left|\sum_{\substack{d | q \\ d<2\pi|x|}}  \frac{\mu(q/d)}{e^{-2\pi i x / d} - 1}\right| \leq \sum_{j = 1}^{\lfloor 2\pi|x| \rfloor} \left| \frac{1}{e^{-2\pi i x / j} - 1} \right| = \O_x(1).
$$
By a similar argument, the two $d < 2\pi|x|$ sums in \eref{eq:d>rSum} are also $\O_x(1)$. Thus, together \eref{eq:SumSeparated} and \eref{eq:d>rSum} imply
$$
\c_q(x) = \frac{e^{2\pi i x} - 1}{2\pi i x} \varphi(q) + \O_x(\sigma_0(q)) + \O_x(1),
$$
which completes the proof.
\end{proof}
This theorem implies the following generalization of \corref{cor:divisorCor}:
\begin{corollary}
\label{cor:divisorCor2}
The generalized divisor function
\begin{equation}
\widetilde{\sigma}_k(x) \defeq x^k\zeta(k+1)\sum_{q \geq 1} \frac{\c_q(x)}{q^{k+1}}
\label{eq:divisorExtension}
\end{equation}
converges absolutely for all $x \in \Z$ if $k > 0$ and for all $x \in \R$ if $k > 1$, but diverges for all $x \in \R \backslash \Z$ if $k = 1$.
\end{corollary}
\begin{proof}
If $x \in \Z$, then $\widetilde{\sigma}_k(x) = \sigma_k(x)$, which converges absolutely when $k > 0$ by the argument after \eref{eq:divisorFunction}. If $x \in \R$, then for $k > 1$,
$$
\sum_{q \geq 1}\left| \frac{\c_q(x)}{q^{k+1}} \right| \leq \sum_{q \geq 1} \frac{\varphi(q)}{q^{k+1}} = \frac{\zeta(k)}{\zeta(k + 1)} < \infty.
$$
If $x \in \R \backslash \Z$ and $k = 1$, however, then it follows from \thmref{thm:maintheorem} that
$$
\sum_{q \geq 1} \frac{\c_q(x)}{q^2} = \sum_{1 \leq q < 2\pi|x|} \frac{\c_q(x)}{q^2} + \frac{e^{2\pi i x} - 1}{2\pi i x} \sum_{q > 2\pi| x|} \frac{\varphi(q)}{q^2} + \sum_{q > 2\pi|x|} \frac{\O_x(\sigma_0(q))}{q^2}.
$$
The sum in the first term is obviously finite, and the series in the last term converges absolutely because $\sum_{q \geq 1} \frac{\sigma_0(q)}{q^2} = \zeta(2)^2$. However, the series in the middle term diverges because
$$
\sum_{q > 2\pi|x|} \frac{\varphi(q)}{q^2}  + \O_x(1) = \sum_{q \geq 1} \frac{\varphi(q)}{q^2} = \lim_{k \rightarrow 1^+} \frac{\zeta(k)}{\zeta(k + 1)}.
$$
Therefore, $\widetilde{\sigma}_1(x)$ diverges for all $x \in \R \backslash \Z$.
\end{proof}
This behavior is surprising. Although extraordinarily well-behaved on integers, any non-integral $\epsilon > 0$ addition to an integer argument causes $\widetilde{\sigma}_1(x)$ to diverge (albeit only logarithmically, as our proof of \corref{cor:divisorCor2} shows). Indeed, this is plainly true for any function $\widetilde{f}(x)$ that has $\widehat{f}(q) = \Omega(q^{-2})$ in its generalized Ramanujan expansion \eref{eq:RamanujanExpansionCont} (assuming it exists). This includes, for example, the generalizations of both Ramanujan's and Hardy's expansions of the zero function \cite{Hardy1, Ramanujan1},
$$
\widetilde{0}_{\text{Ram}}(x) \defeq \sum_{q \geq 1}\frac{\c_q(x)}{q} \quad \text{and} \quad \widetilde{0}_{\text{Har}}(x) \defeq \sum_{q \geq 1}\frac{\c_q(x)}{\varphi(q)},
$$
as well as the generalization of the number-of-divisors function \cite{Ramanujan1},
\begin{equation}
\widetilde{\sigma}_0(x) \defeq -\sum_{q \geq 1}\frac{\log q}{q}\c_q(x).
\label{eq:numberofdivisorgeneralization}
\end{equation}
It is evident from \thmref{thm:maintheorem} that these series converge if and only if $x \in \Z$. 

For $\sigma_0(n)$ and $\sigma_1(n)$, we are able to construct continuous and absolutely convergent Ramanujan expansions that interpolate these functions, although only locally on $\R$. We describe these constructions in the next section.

\section{Real Variable Ramanujan Expansions of $\sigma_0(n)$ and $\sigma_1(n)$}
\label{sec:divisorExtension}

Fix $\alpha \in \Z \backslash \{0\}$ and denote its divisors by $d_1, \dots, d_N$, where $N = \sigma_0(\alpha)$. Define $P_{\alpha}(x)$ as the degree-$N$ polynomial whose roots are $d_1, \dots, d_N$:
\begin{equation}
P_{\alpha}(x) \defeq \prod_{i=1}^N(x - d_i) = \sum_{n=0}^N a_n x^n.
\label{eq:divisorPolynomial}
\end{equation}
The polynomial coefficients $a_n$ are given by
\begin{equation}
a_n = (-1)^{N-n}e_{N-n}(d_1,\dots, d_N),
\label{eq:coefficients}
\end{equation}
where $e_{j}$ is the $j$th elementary symmetric polynomial in the divisors of $\alpha$:
$$
e_{j}(d_1,\dots, d_N) \defeq \sum_{1 \leq i_1 < \cdots < i_j \leq N} d_{i_1} \cdots d_{i_j}.
$$
Since $P_{\alpha}(d) = 0$ for each $d \in \{d_1, \dots, d_N\}$, we have the linear system of equations
$$
\begin{cases}
P_{\alpha}(d_1) = a_Nd_1^N + \cdots + a_1d_1 + a_0 = 0\\
P_{\alpha}(d_2) = a_Nd_2^N + \cdots + a_1d_2 + a_0 = 0\\
\hspace{8em}\vdots\\
P_{\alpha}(d_N) = a_Nd_N^N + \cdots + a_1d_N + a_0 = 0.
\end{cases}
$$
In fact, we have for any natural number $k$ the more general system
$$
\begin{cases}
a_Nd_1^{k + N} + \cdots + a_1d_1^{k+1} + a_0d_1^{k} = 0\\
a_Nd_2^{k + N} + \cdots + a_1d_2^{k+1} + a_0d_2^{k} = 0\\
\hspace{7em}\vdots\\
a_Nd_N^{k + N} + \cdots + a_1d_N^{k+1} + a_0d_N^{k} = 0.
\end{cases}
$$
Adding the equations in this second system implies
$$
a_N\left(d_1^{k + N} + \cdots + d_N^{k + N}\right) + \cdots + a_0\left(d_1^{k} + \cdots + d_N^{k}\right) = 0,
$$
which is equivalent to
$$
a_N\sigma_{k + N}(\alpha) + \cdots + a_0\sigma_{k}(\alpha) = 0.
$$
By \eref{eq:coefficients} and the fact that no divisor of $\alpha$ is zero, $a_0 = d_1\cdots d_N \neq 0$. Consequently,
\begin{equation}
\sigma_k(\alpha) = -\frac{1}{a_0}\big( a_N\sigma_{k + N}(\alpha) + \cdots + a_1\sigma_{k + 1}(\alpha) \big).
\label{eq:newformulation}
\end{equation}
Now, on substituting equations \eref{eq:divisorFunction} and \eref{eq:coefficients} into \eref{eq:newformulation}, we obtain a new and nontrivial formula for the divisor function ($k \geq 0$, $\alpha \in \Z \backslash \{0\}$):
\begin{equation}
\sigma_k(\alpha) = \sum_{q \geq 1} \sum_{j=0}^{N - 1} \left[ \frac{(-1)^{j+1} e_j(d_1,\dots, d_N)}{d_1\cdots d_N}  \cdot \frac{\zeta(N + k + 1 - j)\alpha^{N + k - j}}{q^{N + k + 1 - j}}\right]c_q(\alpha).
\label{eq:newDivisorSum}
\end{equation}
This equation constitutes a Ramanujan expansion of $\sigma_k(n)$, albeit an ``impure'' one as the Ramanujan coefficients depend on the argument $\alpha$.

We now take $k = 1$ in \eref{eq:newformulation}, and redefine $\widetilde{\sigma}_1(x)$ in \eref{eq:sumofdivisorcontinuous} by the continuous generalization of \eref{eq:newformulation}:\footnote{We explicitly include in $\widetilde{\sigma}_1(x, \alpha)$ a functional dependence on $\alpha$ to highlight the fact that $\alpha$ is implicit in the $a_n$ coefficients.}
\begin{equation}
\widetilde{\sigma}_1(x, \alpha) \defeq -\frac{1}{a_0}\big( a_N\widetilde{\sigma}_{N + 1}(x) + \cdots + a_1\widetilde{\sigma}_{2}(x) \big).
\label{eq:divisorFunctionContinuationDef}
\end{equation}
\corref{cor:divisorCor2} proves that the $\widetilde{\sigma}_{i + 1}(x)$ functions on the right side converge absolutely for all $x \in \R$ because $i > 0$. It follows that $\widetilde{\sigma}_1(x, \alpha)$ converges absolutely for all $x \in \R$. On substituting equations \eref{eq:divisorExtension} and \eref{eq:coefficients} into \eref{eq:divisorFunctionContinuationDef}, we obtain the continuous analogue of \eref{eq:newDivisorSum} with $k = 1$:
\begin{equation}
\widetilde{\sigma}_1(x, \alpha) = \sum_{q \geq 1} \sum_{j=0}^{N - 1} \left[ \frac{(-1)^{j+1} e_j(d_1,\dots, d_N)}{d_1\cdots d_N}  \cdot \frac{\zeta(N+2-j)x^{N+1-j}}{q^{N+2-j}}\right] \c_q(x).
\label{eq:newDivisorSumTwo}
\end{equation} 
Unlike \eref{eq:sumofdivisorcontinuous}, this constitutes a continuous and absolutely convergent Ramanujan expansion for $\widetilde{\sigma}_1(x, \alpha)$.

By a similar construction with $k = 0$ in \eref{eq:newformulation}, but in which we replace $\sigma_1(\alpha)$ by \eref{eq:newDivisorSumTwo}, we redefine $\widetilde{\sigma}_0(x)$ in \eref{eq:numberofdivisorgeneralization} to
\begin{equation}
\widetilde{\sigma}_0(x, \alpha) \defeq -\frac{1}{a_0}\big( a_N\widetilde{\sigma}_{N}(x) + \cdots + a_2\widetilde{\sigma}_{2}(x)  +  a_1\widetilde{\sigma}_{1}(x, \alpha) \big).
\label{eq:numberdivisorFunctionContinuationDef}
\end{equation}
On inserting equations \eref{eq:divisorExtension} and \eref{eq:coefficients} into \eref{eq:numberdivisorFunctionContinuationDef}, we obtain an absolutely convergent and continuous Ramanujan expansion for $\widetilde{\sigma}_0(x, \alpha)$, akin to $\widetilde{\sigma}_1(x, \alpha)$ in \eref{eq:newDivisorSumTwo}.

On what subset of $\R$ do \eref{eq:newDivisorSumTwo} and \eref{eq:numberdivisorFunctionContinuationDef} correctly interpolate $\sigma_1(n)$ and $\sigma_0(n)$, respectively? The answer, unfortunately, is on a local subset only---namely, the open neighborhood $(\alpha - 1, \alpha + 1) \subset \R$. This follows because \eref{eq:newformulation} presupposes a \emph{fixed} integer $\alpha$ by which the coefficients in \eref{eq:newformulation} are defined. It is plain, then, that only in a neighborhood of $\alpha$ containing no other integers will the equality \eref{eq:newformulation} necessarily hold, since the polynomial coefficients in \eref{eq:divisorPolynomial} for any integer other than $\alpha$ are generally different from those in $P_\alpha(x)$.

So while we have shown that there exist continuous and absolutely convergent Ramanujan expansions that interpolate $\sigma_0(n)$ and $\sigma_1(n)$, they do so only locally in the neighborhood around a given integer. This begs the question of whether there exists a continuous and absolutely convergent Ramanujan expansion that interpolates $\sigma_1(n)$ (and perhaps even $\sigma_0(n)$) \emph{globally} on $\R$, or at least on the positive reals.

\bibliographystyle{amsplain}
\bibliography{DivisorReferences}

\end{document}